\documentclass[11pt,twoside]{amsart}
\usepackage[all]{xy}
\usepackage{enumerate}
\usepackage{palatino}
\usepackage{bbm}
\usepackage{epsfig}
\usepackage{amsmath, amssymb, amsthm}

\theoremstyle{plain}
\newtheorem{thm}{Theorem}[section]
\newtheorem{prop}[thm]{Proposition}
\newtheorem{lemma}[thm]{Lemma}
\newtheorem{cor}[thm]{Corollary}

\theoremstyle{remark}
\newtheorem{rmk}[thm]{Remark}
\theoremstyle{definition}
\newtheorem{defn}[thm]{Definition}

\newtheorem*{lemma*}{Lemma}

\renewcommand{\AA}{\mathbb{A}}

\newcommand{\CC}{\mathbb{C}}

\newcommand{\EE}{\mathbb{E}}
\newcommand{\FF}{\mathbb{F}}

\newcommand{\NN}{\mathbb{N}}

\newcommand{\QQ}{\mathbb{Q}}
\newcommand{\RR}{\mathbb{R}}

\newcommand{\ZZ}{\mathbb{Z}}
\newcommand{\kk}{\mathsf{k}}

\newcommand{\calO}{\mathcal{O}}
\newcommand{\calA}{\mathcal{A}}

\newcommand{\frakm}{\mathfrak{m}}

\newcommand{\calE}{\mathcal{E}}

\newcommand{\Tor}{\operatorname{Tor}}

\newcommand{\Ext}{\operatorname{Ext}}

\newcommand{\BPGL}{\mathsf{BPGL}}

\newcommand{\BP}{\mathsf{BP}}

\newcommand{\Spec}{\operatorname{Spec}}

\renewcommand{\top}{\operatorname{top}}

\newcommand{\Frac}{\operatorname{Frac}}
\newcommand{\MGL}{\mathsf{MGL}}
\newcommand{\KGL}{\mathsf{KGL}}
\newcommand{\kgl}{\mathsf{kgl}}

\newcommand{\MU}{\mathsf{MU}}

\newcommand{\comp}{\,\widehat{_2}}
\renewcommand{\top}{{\operatorname{top}}}

\newcommand{\MV}{\operatorname{\mathcal{M\kern -.1em V}}}

\newcommand{\braces}[1]{\left\{ #1\right\}}

\newcommand{\tensor}{\otimes}
\newcommand{\cotensor}{\mathbin{\Box}}

\newcommand{\smsh}{\wedge}

\newcommand{\coleq}{\mathrel{\mathop:}=}

\begin{document}

\bibliographystyle{amsalpha} 

\title{Motivic invariants of $p$-adic fields}
\author{Kyle M.~Ormsby}
\email{ormsby@math.mit.edu}
\thanks{This research was partially supported by NSF grant DMS-0602191.}

\begin{abstract}
We provide a complete analysis of the motivic Adams spectral sequences converging to the bigraded coefficients of the 2-complete algebraic Johnson-Wilson spectra $\BPGL\langle n\rangle$ over $p$-adic fields.  These spectra interpolate between integral motivic cohomology ($n=0$), a connective version of algebraic $K$-theory ($n=1$), and the algebraic Brown-Peterson spectrum ($n=\infty$).  We deduce that, over $p$-adic fields, the 2-complete $BPGL\langle n\rangle$ splits over 2-complete $\BPGL\langle 0\rangle$, implying that the slice spectral sequence for $BPGL$ collapses.

This is the first in a series of two papers investigating motivic invariants of $p$-adic fields, and it lays the groundwork for an understanding of the motivic Adams-Novikov spectral sequence over such base fields.
\end{abstract}

\maketitle

\setcounter{tocdepth}{1}
\tableofcontents

\section{Introduction}
\label{sec:intro}
This paper initiates a project to determine algebro-geometric invariants of $p$-adic fields via the methods of stable homotopy theory.  The technology for such an endeavor resides in the Morel-Voevodsky motivic homotopy theory \cite{MV}, and in the stabilizations thereof \cite{VoevICM, PoSMod, Jardine, DOR}.  The techniques here are natural generalizations of those used over an algebraically closed field in \cite{HKO}, but the phenomena observed are more nuanced because of the arithmetically richer input.

Presently, we will concern ourselves with the bigraded coefficients of $\BPGL\langle n\rangle\comp$ (the 2-complete $n$-th algebraic Johnson-Wilson spectrum at the prime 2, cf.~Definition \ref{defn:BPGLn})  over a $p$-adic field, $p>2$.  A sequel to this work will use these results to provide information about a motivic Adams-Novikov spectral sequence converging to stable motivic homotopy groups of the 2-complete sphere spectrum over a $p$-adic field \cite{motAlpha, thesis}. 
Our main computational tool in all cases is the motivic Adams spectral sequence.

Our grading conventions will follow those in \cite{HKO}, where the $(m+n\alpha)$-sphere $S^{m+n\alpha}$ is the smash product $(S^1)^{\smsh m}\smsh (\AA^1\smallsetminus 0)^{\smsh n}$.  The wildcard $\star$ will refer to bigradings of the form $m+n\alpha$, $m,n\in \ZZ$, and if $E$ is a motivic spectrum then its (bigraded) coefficients are $E_\star = \pi_\star E$.

In \S\ref{sec:comod}, we define and establish basic properties of $\BPGL\langle n\rangle$ and identify its mod 2 motivic homology as a comodule over the dual motivic Steenrod algebra $\calA_\star$.  In \S\ref{sec:Ext}, we run appropriate filtration spectral sequences that determine the $E_2$-terms of the motivic Adams spectral sequences converging to 2-complete coefficients.  We then analyze these spectral sequences in \S\ref{sec:motASS} in order to fully determine the bigraded coefficient rings $\pi_\star\BPGL\langle n\rangle\comp$.  The computation of $\pi_\star BPGL\comp$, combined with motivic Landweber exactness, permits a description of the $\BPGL\comp$ Hopf algebroid producing a computation of the $E_2$-term of the motivic Adams-Novikov spectral sequence over a $p$-adic field; see Theorem \ref{thm:ANSSE2}.

In order to follow this program, we use the rest of this introduction to review fundamental input from motivic cohomology. In \S\ref{sec:pAdic}, we describe our conventions for $p$-adic fields and review arithmetic input making explicit computations possible.

\subsection*{Motivic homology and the dual motivic Steenrod algebra over a field}
Let $k^M_*$ denote the mod 2 reduction of Milnor $K$-theory $K^M_*$; let $H$ denote the mod 2 motivic cohomology spectrum.  The main result of \cite{motCohom} determines the motivic cohomology of $\Spec(\kk)$, while \cite{redPow} and \cite{motEM} determine stable cohomology operations on mod 2 motivic cohomology.

\begin{thm}[{\cite{motCohom}}]\label{thm:motCohom}
Mod 2 motivic cohomology of $\Spec(\kk)$ takes the form
\[
	H^\star(\Spec(\kk);\ZZ/2) = k^M_*(\kk)[\tau]
\]
where $|k^M_1(\kk)| = \alpha$ and $|\tau| = -1+\alpha$.
\qed
\end{thm}

\begin{thm}[\cite{redPow,motEM}]\label{thm:Steen}
The \emph{motivic Steenrod algebra} is the algebra of stable operations on $H$,
\[
	\calA^\star \coleq H^\star H.
\]
The motivic Steenrod algebra is generated by $\beta$ and $P^i$, $i\ge 0$.
\qed
\end{thm}

$\calA^\star$ has the structure of a Hopf algebroid over $H^\star$.  (See \cite[Appendix A1]{Rav} for the theory of Hopf algebroids.)  In this paper, we will be more concerned with the dual to the motivic Steenrod algebra, $\calA_\star = H_\star H$ which is a Hopf algebroid over $H_\star = H^{-\star}$.

\begin{thm}[{\cite{redPow,motEM}}]\label{thm:dualSteen}
The dual motivic Steenrod algebra is isomorphic to
\[
	H_\star[\tau_0,\tau_1,\ldots,\xi_1,\xi_2,\ldots]/(\tau_i^2 - \tau\xi_{i+1} - \rho(\tau_{i+1} + \tau_0\xi_{i+1})).
\]
Here $\tau$ is the generator of $H_{1-\alpha}= H^0(\Spec(\kk);\ZZ/2(1))$, $\rho$ is the class of $-1$ in $H_{-\alpha} = k^M_1(\kk) = \kk^\times/(\kk^\times)^2$, $|\tau_i| = (2^i-1)(1+\alpha)+1$, and $|\xi_i| = (2^i-1)(1+\alpha)$.

The Hopf algebroid structure on $\calA_\star$ is specified by the following: elements of $H_{0+*\alpha} = k^M_*(\kk)$ are primitive, and
\begin{equation}\label{eqn:HopfStr}
\begin{aligned}
	\eta_L \tau	&= \tau\\
	\eta_R \tau	&= \tau + \rho\tau_{0}\\
	\Delta \xi_{k}	&= \sum_{i=0}^{k}\xi_{k-i}^{2^{i}}\tensor \xi_{i}\\
	\Delta \tau_{k}	&= \tau_{k}\tensor 1 + \sum_{i=0}^{k}\xi_{k-i}^{2^{i}}\tensor \tau_{i}.\qed
\end{aligned}
\end{equation}
\end{thm}

\begin{rmk}\label{rmk:MilnorBasis}
It follows that $\calA^\star$ has a Milnor basis of elements of the form $Q_I(r_1,\ldots,r_n)$ as in topology with the degree shift $|Q_n| = 2^n(1+\alpha)-\alpha$; see \cite{Borg}.
\end{rmk}

Certain quotient Hopf algebroids of $\calA_\star$ will be useful in our analysis (see \S\ref{sec:comod}).  The following definition is due to Mike Hill.

\begin{defn}[{\cite{ExtHill}}]\label{defn:En}
Let $\mathcal{E}(n)$, $0\le n < \infty$, denote the quotient Hopf algebroid
\[\begin{aligned}
	\mathcal{E}(n)	&\coleq \mathcal{A}_\star/(\xi_1,\xi_2,\ldots,\tau_{n+1},\tau_{n+2},\ldots)\\
			&= H_\star[\tau_0,\ldots,\tau_n]/(\tau_i^2-\rho\tau_{i+1}\mid 0\le i<n)+(\tau_n^2).
\end{aligned}\]
If $n=\infty$, let
\[\begin{aligned}
	\mathcal{E}(\infty)	&\coleq \mathcal{A}_\star/(\xi_1,\xi_2,\ldots)\\
				&= H_\star[\tau_0,\tau_1,\ldots]/(\tau_i^2-\rho\tau_{i+1}\mid 0\le i).
\end{aligned}\]
\end{defn}

The Hopf algebroid $\calE(n)$ is dual to the sub-Hopf algebroid of $\calA^\star$ generated by the Milnor elements $Q_i$, $i\le n$.

\subsection*{The motivic Adams spectral sequence}
Our primary means of computation is the motivic Adams spectral sequence (mASS).  This spectral sequence first appeared in Morel's work on connectivity and stable motivic $\pi_0$ \cite{MorelConn,MorelPi0} and has since been used over algebraically closed fields by Hu-Kriz-Ormsby \cite{HKO} (more accurately, their work focuses on the application of the motivic Adams-Novikov spectral sequence) and, independently, Dugger-Isaksen \cite{DI}.  Hopkins-Morel (unpublished) knew that the motivic Adams spectral sequence (at the prime 2) converges to $(2,\eta)$-completions and Hu-Kriz-Ormsby \cite{KOAdams} both prove this result and show that $\eta$-completion is unnecessary when $cd_2(\kk[i])<\infty$.  This condition holds for $\kk=F$ a $p$-adic field, so we have the following:

\begin{thm}[\cite{KOAdams}]\label{thm:conv}
Fix a $p$-adic field $F$ (see \S\ref{sec:pAdic}) and let $X$ be a cell spectrum of finite type.  Then the $E_2$-term of the mASS is
\[
	E_2^{*,\star} = \Ext_{\calA_\star}(H_\star,H_\star X)
\]
and the mASS converges to $\pi_\star X\comp$ where permanent cycles in tri-degree $(s,m+n\alpha)$ represent elements of $\pi_{m+n\alpha-s}X\comp$.
\qed
\end{thm}

A word on the grading of the mASS will make computations easier to follow:  The mASS is tri-graded.  We denote the $r$-th page of the mASS by $E_r^{*,\star}$ where the first $*$ is an integer called the \emph{homological degree}, and the second $\star$ is a bigrading of the form $m+n\alpha$ called the \emph{motivic degree}.  For a tri-grading $(s,m+n\alpha)$, we call the bigrading $m+n\alpha - s = (m-s)+n\alpha$ the \emph{total motivic degree} or \emph{Adams grading}; sometimes Adams grading will also refer to the tri-degree $(s,m+n\alpha-s)$.  The differentials in the mASS take the form
\[
	d_r:E_r^{s,m+n\alpha}\to E_r^{s+r,m+n\alpha+r-1}.
\]
In other words, the $r$-th differential increases homological degree by $r$ and decreases total motivic degree by 1.

\begin{rmk}
This paper concerns itself with mASS computations of the bigraded coefficients of the 2-complete $\BPGL\langle n\rangle$ over $p$-adic fields.  Its sequel \cite{motAlpha} analyzes the motivic ANSS over $p$-adic fields, in particular a motivic analogue of the alpha family in that setting.
\end{rmk}

\subsection*{Acknowledgments}
This paper represents the first half of my thesis and it is a genuine pleasure to thank my advisor, Igor Kriz, for his input and help.  I would also like to thank Mike Hill and Paul Arne {\O}stv{\ae}r for their encouragement and interest over the summer of 2009.  Finally, I would like to thank the anonymous referee for numerous stylistic improvements, a correction to the proof of Theorem \ref{thm:cohomkgl}, a strengthening of Proposition \ref{prop:relns}, and a streamlined method of proof for Theorem \ref{thm:BPGL} that avoided dependence on $K$-theory computations.

\section{Arithmetic input from $p$-adic fields}
\label{sec:pAdic}
A \emph{$p$-adic field} is a complete discrete valuation field of characteristic 0 with finite characteristic $p$ residue field.  It is well-known that every $p$-adic field is a finite extension of the $p$-adic rationals $\QQ_p$.  A good reference for the basic theory is \cite{Cassels}.

Let $v:F\to \ZZ\cup\infty$ denote the valuation on $F$.  $F$ has a \emph{ring of integers} $\calO \coleq \braces{x\in F\mid v(x)\ge 0}$.  $\calO$ is a domain with $F = \Frac \calO$, the field of fractions of $\calO$.  Moreover, $\calO$ is a local ring with maximal ideal $\frakm \coleq \braces{x\in F\mid v(x)\ge 1}$.  A \emph{uniformizer} of $F$ is an element $\pi\in F$ such that $v(\pi) = 1$; note that for any choice of uniformizer $\pi$, $(\pi) = \frakm$.

The \emph{residue field} of $F$ is
\[
	\FF \coleq \calO/\frakm.
\]
Let $q = p^m = |\FF|$ denote the \emph{residue order} of $F$.

As a consequence of Hensel's lemma (see, e.g., \cite[Lemma 3.1]{Cassels}), the units of a $p$-adic field $F$ are equipped with a \emph{Teichm\"{u}ller lift} $\FF^\times\hookrightarrow F^\times$.  Identifying $\FF^\times$ with its image in $F^\times$, we have
\[
	F^\times = \pi^\ZZ \times \FF^\times \times (1+\frakm).
\]

\begin{cor}\label{cor:sqCl}
Let $F$ be a $p$-adic field, $p>2$, with chosen uniformizer $\pi$ and choose $u$ to be a nonsquare in the Teichm\"{u}ller lift $\FF^\times$.  Then
\[
	F^\times/(F^\times)^2 = \pi^{\ZZ/2} \times u^{\ZZ/2}.
\]
When $q = |\FF^\times| \equiv 3\pod{4}$, we may choose $u$ to be $-1$; when $q\equiv 1\pod{4}$, the image of $-1$ in $F^\times/(F^\times)^2$ is zero.  (If $p=2$, then $(1+\frakm)/(1+\frakm)^2 \ne 0$.)
\qed
\end{cor}

\begin{rmk}\label{conv}
The structure of $p$-adic fields differs in the cases $p=2$ and $p>2$:  for instance, $|\QQ_2^\times/(\QQ_2^\times)^2| = 8$ while $p$-adic fields have $|F^\times/(F^\times)^2|=4$ for every $p>2$.  In order to avoid a great many minor modifications, we will only deal with $p$-adic fields for which $p>2$ in this paper.  Henceforth, the term $p$-adic field will only refer to \emph{nondyadic} $p$-adic fields; moreover, the letter $F$ will always refer to a $p$-adic field unless stated otherwise.
\end{rmk}

Every discretely valued field $(E,v)$ with residue field $\EE$ comes equipped with a \emph{tame symbol}
\[
	\left(\frac{\cdot,\cdot}{E}\right):E^\times\times E^\times\to \EE^\times
\]
defined by the formula
\[
	\left(\frac{x,y}{E}\right) = (-1)^{v(x)v(y)} x^{v(y)} y^{-v(x)}\mod \frakm.
\]

\begin{lemma}[{\cite[Lemma 11.5]{MilnorK}}]\label{lem:tame}
The tame symbol is a Steinberg symbol and hence induces a homomorphism $K^M_2(E)\to \EE^\times = K^M_1(\EE)$.
\qed
\end{lemma}

As a consequence of Lemma \ref{lem:tame} and \cite[Example 1.7]{MilnorQuad}, we can determine the mod 2 Milnor $K$-theory of a $p$-adic field, a result presumably well-known to those who study such objects.

\begin{prop}\label{prop:kMpAdic}
Fix a $p$-adic field $F$, a uniformizer $\pi$, and a nonsquare $u\in\FF^\times$.  As a $\ZZ$-graded $\ZZ/2$-algebra,
\begin{equation}\label{eqn:kMpAdic}
	k^M_*(F) =
	\begin{cases}
		\ZZ/2[\braces{u},\braces{\pi}]/(\braces{u}^2,\braces{\pi}^2)	&\text{if }q\equiv 1\pod{4},\\
		\ZZ/2[\braces{u},\braces{\pi}]/(\braces{u}^2,\braces{\pi}(\braces{u}-\braces{\pi}))	&\text{if }q\equiv 3\pod{4}
	\end{cases}
\end{equation}
where $|\braces{\pi}| = |\braces{u}| = 1$.
\end{prop}
\begin{proof}
Abusing notation, we will write $x$ for $\braces{x}\in K^M_1(F)$ or $k^M_1(F)$ whenever the context does not admit confusion.

Since $k^M_1(F)=F^\times/(F^\times)^2$, Corollary \ref{cor:sqCl} implies that
\[
	k^M_1(F) = \pi^{\ZZ/2}\times u^{\ZZ/2}.
\]
Moreover, in \cite[Example 1.7(2)]{MilnorQuad} Milnor shows that $k^M_2(F)$ has dimension 1 as a $\ZZ/2$-vector space.  By the same reference, $k^M_n(F) = 0$ for all $n\ge 3$.

We still must determine the multiplicative structure of $k^M_*(F)$, which amounts to determining the products $u^2, u \pi, \pi^2\in k^M_2(F)$.  First note that
\[
	\left(\frac{u, \pi}{F}\right) = (-1)^0 u^1 \pi^0 = u\in K^M_1(\FF),
\]
which reduces to the nontrivial generator $u$ of $k^M_1(\FF)$.  By Lemma \ref{lem:tame}, it follows that $u\pi\ne 0\in k^M_2(F)$.

The argument above also proves that, after reduction mod 2, the tame symbol is an isomorphism $k^M_2(F)\to k^M_1(\FF)$.  Hence to compute the products $u^2$ and $\pi^2$, it suffices to compute
\[
	\left(\frac{u,u}{F}\right)\quad\text{and}\quad \left(\frac{\pi,\pi}{F}\right).
\]
These symbols are $1$ and $-1$, respectively, so $u^2 = 0\in k^M_1(F)$ while $\pi^2$ is nontrivial iff $q\equiv 3\pod{4}$.  This determines the multiplicative structure given in (\ref{eqn:kMpAdic}).
\end{proof}

\begin{thm}\label{thm:pAdicDualSteenrod}
Over a $p$-adic field $F$, the coefficients of mod 2 motivic homology are
\[
	H_\star = k^M_*(F)[\tau]
\]
where $|\tau| = 1-\alpha$, $|k^M_n(F)| = -n\alpha$, and $k^M_*(F)$ has the form given in Proposition \ref{prop:kMpAdic}.

The dual motivic Steenrod algebra has the form
\[
	\calA_\star = H_\star[\tau_0,\tau_1,\ldots,\xi_1,\xi_2,\ldots]/(\tau_i^2 - \tau\xi_{i+1} - \rho(\tau_{i+1} + \tau_0\xi_{i+1})).
\]

The class $\rho$ is trivial iff $q\equiv 1\pod{4}$.  In this case,
\[
	\calA_\star = H_\star[\tau_0,\tau_1,\ldots,\xi_1,\xi_2,\ldots]/(\tau_i^2 - \tau\xi_{i+1}) \cong \calA^\CC_\star\tensor_{H^\CC_\star} k^M_*(F)
\]
where $(H^\CC_\star,\calA^\CC_\star)$ is the dual motivic Steenrod algebra over $\CC$, which has the structure
\[\begin{aligned}
	H^\CC_\star	&= \ZZ/2[\tau],\\
	\calA^\CC_\star	&= \ZZ/2[\tau,\tau_0,\tau_1,\ldots,\xi_1,\xi_2,\ldots]/(\tau_i^2 - \tau\xi_{i+1}).
\end{aligned}\]
\end{thm}
\begin{proof}
Most of the theorem is a concatenation of results in Theorems \ref{thm:motCohom} and \ref{thm:dualSteen} and Proposition \ref{prop:kMpAdic}.  The form of $(H^\CC_*,\calA^\CC_*)$ is obvious after noting that $k^M_*(\CC)$ is trivial outside of degree 0.  The class $\rho$ is trivial iff $-1$ is a square in $\FF^\times = \FF_q^\times$; it is standard that this is the case iff $q\equiv 1\pod{4}$.
\end{proof}

The structure of $H_\star$ over $F$ is depicted in Figure \ref{fig:HF}.  Here the horizontal axis measures the $\ZZ$-component of the motivic bigrading, while the vertical axis measures the $\ZZ\alpha$-component.  Each ``diamond" shape is a copy of $k^M_*(F)$, and the diagonal arrows of slope $-1$ represent $\tau$-multiplication.

\begin{figure}
\centering
\includegraphics[width=2in]{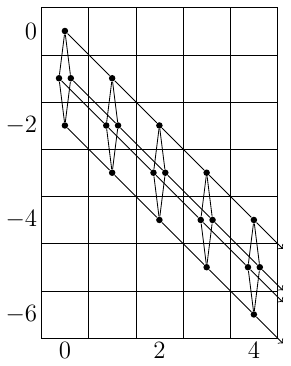}
\caption{Mod 2 motivic homology over $F$}\label{fig:HF}
\end{figure}

\section{Comodoules over the dual motivic Steenrod algebra}
\label{sec:comod}
In this section, we work over a general characteristic 0 field $\kk$.  Recall the algebraic Brown-Peterson spectrum $\BPGL$ constructed by Hu-Kriz and Vezzosi in \cite{realAlgCobord,Vez}.  (We only consider $\BPGL$ at the prime 2 in this paper.)  There are canonical elements $v_1,v_2,\ldots\in \BPGL_\star$ that appear in dimensions $|v_i| = (2^{i}-1)(1+\alpha)$.  Let $v_0=2\in \BPGL_0$.  These elements are the images of $v_i\in \BP_{2(2^i-1)}$ under the Lazard ring isomorphism $\MU_*\to \MGL_{*(1+\alpha)}$.

\begin{defn}
\label{defn:BPGLn}
For $0\le n$, the \emph{algebraic Johnson-Wilson spectra} are defined to be
\[
\BPGL\langle n\rangle \coleq \BPGL/(v_{n+1},v_{n+2},\ldots).
\]
\end{defn}

These quotients are well-defined since algebraic cobordism is $E_\infty$.  They fit into cofiber sequences
\begin{equation}\label{eqn:cofib}
	\Sigma^{|v_n|}\BPGL\langle n\rangle\xrightarrow{v_n} \BPGL\langle n\rangle\to \BPGL\langle n-1\rangle.
\end{equation}
By convention, we write $\BPGL\langle \infty\rangle = \BPGL$.

The study of the algebraic Johnson-Wilson spectra should be motivated by the natural role they play when $n=0,1$, and $\infty$.

\begin{thm}[Hopkins-Morel]\label{thm:BPGL0}
After 2-completion, $\BPGL\langle 0\rangle$ is the 2-complete integral motivic cohomology spectrum,
\[
	\BPGL\langle 0\rangle\comp = H\ZZ_{2}.
\]
\end{thm}

Hopkins defined $\BPGL\langle 1\rangle$ as a motivic analogue of connective $K$-theory, and we will sometimes write $\kgl$ instead of $\BPGL\langle 1\rangle$.  Throughout the rest of this section, let $\KGL$ denote the $2$-localization of the $(1+\alpha)$-periodic algebraic $K$-theory spectrum.  The following theorem relates the coefficients of $\kgl$ with established objects of interest, the algebraic $K$-groups of the ground field.

\begin{thm}\label{thm:v1tors}
Let ${}_{v_1}\kgl_\star$ denote the $v_1$-power torsion in the coefficients of $\kgl$, i.e., the elements $x\in \kgl_\star$ such that there exists $n\in \NN$ such that $v_1^nx=0\in \kgl_\star$.  (We will refer to these elements simply as $v_1$-torsion.)  Then there is an exact sequence
\[
	0\to {}_{v_1}\kgl_\star \to \kgl_\star\to \KGL_\star.
\]
Moreover, if $\overline{\KGL_\star}$ denotes the subalgebra of $\KGL_\star$ consisting of elements in degree $m+n\alpha$, $m\ge 0$, then there is a short exact sequence
\[
	0\to {}_{v_1}\kgl_\star\to \kgl_\star\to \overline{\KGL_\star}\to 0.
\]
\end{thm}
\begin{proof}
By the motivic Conner-Floyd theorem \cite{OS}, $\KGL_\star = v_1^{-1}\kgl_\star$.  The first exact sequence is then a basic fact of localization.

Clearly, though, the map $\kgl_\star\to \KGL_\star$ is not surjective since $\KGL_\star$ is $\text{Bott} = v_1$-periodic.  Note, though, that $\KGL_\star$ is generated by $v_1$ of dimension $1+\alpha$ and elements of degree $m+n\alpha$, $m\ge 0$.  (In fact, we could restrict the second collection of generators to degrees $0+n\alpha$, $n\le 0$.)  Again by the motivic Conner-Floyd theorem, it is a straightforward combinatorial check that $\kgl_\star\to \KGL_\star$ is surjective in dimensions $m+n\alpha$, $m\ge 0$.
\end{proof}

\begin{rmk}\label{rmk:algK}
The spectrum $\kgl$ is connective in the sense that $\kgl_{m+n\alpha} = 0$ for all $m<0$.  In general (though, we will see, not for $p$-adic fields) there is a rich class of $v_1$-torsion in $\kgl_\star$, so it is the case that $\kgl_\star$ is bigger than $\KGL_\star$ in its nonvanishing dimensional range.  Still, producing computations of $\kgl_\star$ explicit enough to capture its $v_1$-torsion will determine $\KGL_\star$ in a meaningful dimensional range by the second exact sequence.  In particular,
\[
	(\kgl_\star/{}_{v_1}\kgl_\star)_{m+0\alpha} = \KGL_{m+0\alpha}
\]
for $m\ge 0$, and these groups match the 2-local Quillen $K$-groups of the base field.
\end{rmk}

We now turn to determining the $\calA_\star$-comodule structure of $H_\star\BPGL\langle n\rangle$.  To access these, we will determine the $\calA^\star$-module structure of $H^\star\BPGL\langle n\rangle$.

Recall the Milnor elements $Q_i\in \calA^\star$, $|Q_i| = 2^i(1+\alpha)-\alpha$ from \S\ref{sec:intro}.  The following theorem of Borghesi should appear quite familiar to topologists.

\begin{thm}[{\cite[Proposition 6]{Borg}}]\label{thm:cohomMGL}
The mod 2 motivic cohomology of $\MGL$ takes the form
\[
	H^\star\MGL = (\calA^\star//E(Q_0,Q_1,\ldots))[m_i\mid i\ne 2^n-1]
\]
as an $\calA^\star$-module where $|m_i| = i(1+\alpha)$.
\qed
\end{thm}

\begin{cor}\label{cor:cohomBPGL}
The mod 2 motivic cohomology of $\BPGL$ takes the form
\[
	H^\star\BPGL = \calA^\star//E(Q_0,Q_1,\ldots)
\]
as an $\calA^\star$-module.
\qed
\end{cor}

Recall Definition \ref{defn:En} which defines the $\calA_\star$-algebras $\calE(n)$, $0\le n\le \infty$.  In particular, we have
\begin{equation}\label{eqn:EooE1}
\begin{aligned}
	\calE(\infty)	&= \calA_\star/(\xi_1,\xi_2,\ldots)\\
			&= H_\star[\tau_0,\tau_1,\ldots]/(\tau_i^2-\rho\tau_{i+1}\mid 0\le i),\\
	\calE(n)	&= \calA_\star/(\xi_1,\xi_2,\ldots,\tau_{n+1},\tau_{n+2},\ldots)\\
			&= H_\star[\tau_0,\ldots,\tau_n]/(\tau_i^2 - \rho\tau_{i+1}\mid 0\le i\le n)+(\tau_n^2).
\end{aligned}
\end{equation}
These algebras are dual to $E(Q_0,Q_1,\ldots)$ and $E(Q_0,\ldots,Q_n)$, respectively.

There is a general yoga of passing from $\calA^\star$-module structure on cohomologies to $\calA_\star$-comodule structure on homologies.  Applied to the above situation, we get the following theorem describing the $\calA_\star$-comodule structure on $H_\star\BPGL$.

\begin{thm}\label{thm:HBPGL}
As an $\mathcal{A}_\star$-comodule algebra,
\[
	H_\star\BPGL = \mathcal{A}_\star\cotensor_{\mathcal{E}(\infty)} H_\star.
\]
\qed
\end{thm}

To determine the $\calA_\star$-comodule structure of $H_\star\BPGL\langle n\rangle$ we first determine $H^\star\BPGL\langle n\rangle$ as an $\calA^\star$-module and then apply the same yoga.  Our determination of $H^\star\BPGL\langle n\rangle$ is modeled on the topological calculation \cite{Wilson}.  (Since the first draft of this paper was written, a similar argument for the cohomology of $\kgl$ has appeared in Isaksen-Shkembi \cite[\S5]{IsakShk}.)

\begin{thm}\label{thm:cohomkgl}
As an $\calA^\star$-module algebra,
\[
	H^\star\BPGL\langle n\rangle = \calA^\star//E(Q_0,\ldots,Q_n).
\]
\end{thm}
\begin{proof}
We use the cofiber sequence (\ref{eqn:cofib}) and induction on $n$.  By Theorem \ref{thm:BPGL0}, we know the Theorem holds for $\BPGL\langle 0\rangle$.  Assume it holds for some $n-1\ge 0$ and consider the long exact sequence in cohomology induced by (\ref{eqn:cofib}).  Following the exact argument of \cite[Proposition 1.7]{Wilson}, it suffices to show that $Q_n(1)=0\in H^{2^n(1+\alpha)-\alpha}\BPGL\langle n\rangle$.  To this end, note that $\BPGL\langle n\rangle$ is constructed from $\BPGL$ by killing off spheres of the form $S^{k+\ell\alpha}$ where $k,\ell\ge 2^{n+1}-1$.  Invoking Morel's connectivity theorem and the long exact sequence in homotopy induced by $\BPGL\to\BPGL\langle n\rangle$, we see that this map induces an iso in degrees $m+*\alpha$, $m+1<2^{n+1}$.  The same holds in cohomology, so Corollary \ref{cor:cohomBPGL} implies $Q_n$ dies in $H^\star\BPGL\langle n\rangle$ since $2^n+1<2^{n+1}$.
\end{proof}

Since $\calE(n)$ is dual to $E(Q_0,\ldots,Q_n)$, we have the following theorem.

\begin{thm}\label{thm:Hkgl}
As an $\mathcal{A}_\star$-comodule algebra,
\[
	H_\star\BPGL\langle n\rangle = \mathcal{A}_\star\cotensor_{\mathcal{E}(n)} H_\star.
\]
\qed
\end{thm}

\section{Motivic $\Ext$-algebras}
\label{sec:Ext}
Theorems \ref{thm:HBPGL} and \ref{thm:Hkgl} identify the homology of $\BPGL\langle n\rangle$, $0\le n\le \infty$, in the category of $\mathcal{A}_\star$-comodules.  By Theorem \ref{thm:conv}, these data form the input to the $E_2$-term of the mASS for $\BPGL\langle n\rangle$.  In fact, both $E_2$-terms take the form
\[
	\Ext_{\mathcal{A}_\star}(H_\star,\mathcal{A}_\star\cotensor_{\mathcal{E}(n)}H_\star).
\]

\begin{thm}[{\cite[Theorem A1.3.12]{Rav}}]\label{thm:changeRings}
For $0\le n\le \infty$, the map of Hopf algebroids $(H_\star,\calA_\star)\to (H_\star,\calE(n))$ induces an isomorphism
\[
	\Ext_{\calA_\star}(H_\star,\calA_\star\cotensor_{\calE(n)}H_\star)\cong \Ext_{\calE(n)}(H_\star,H_\star).\qed
\]
\end{thm}

Fix a $p$-adic field $F$ (see Remark \ref{conv}) with residue order $q$.  In this section, we compute $\Ext_{\mathcal{E}(n)}(H_\star,H_\star)$ over $F$; this is the $E_2$-term for the mASS computing $\pi_\star\BPGL\langle n\rangle\,\widehat{_2}$.  This work was antecedent to Hill's paper \cite{ExtHill} in which he performs similar computations over the field of real numbers $\RR$.

Recall that when $q\equiv 1\pod{4}$, Theorem \ref{thm:pAdicDualSteenrod} implies that $\Ext_{\mathcal{E}(n)}(H_\star,H_\star)$ is easily computable in terms of its complex counterpart $\Ext_{\mathcal{E}(n)^\CC}(H^\CC_\star,H^\CC_\star)$.  In fact, $(H^F_\star,\mathcal{E}(n)^F) = (H^\CC_\star,\mathcal{E}(n)^\CC) \tensor_{H^\CC_\star} H^F_\star$ as Hopf algebroids, so, by change of base,
\[
	\Ext_{\mathcal{E}(n)^F}(H^F_\star,H^F_\star) = \Ext_{\mathcal{E}(n)^\CC}(H^\CC_\star,H^\CC_\star) \tensor_{H^\CC_\star} H^F_\star
\]
when $q\equiv 1\pod{4}$.  Moreover, since $\rho = 0$ over $\CC$, $\mathcal{E}(n)^\CC = \mathcal{E}(n)^\top \tensor_{\ZZ/2} H^\CC_\star$.  Here $\mathcal{E}(n)^\top$ is the analogous quotient of the topological dual Steenrod algebra, but degree-shifted so that elements usually in degree $2m$ appear in dimension $m(1+\alpha)$.  Hence, again by change of base, we can compute the $E_2$-term of the mASS for $\BPGL\langle n\rangle^\CC$.  To be precise,
\[\begin{aligned}
	\Ext_{\calE(n)^\CC}(H^\CC_\star,H^\CC_\star)
	&= \Ext_{\calE(n)^{\top}}(H^{\top}_\star,H^{\top}_\star)\tensor_{\ZZ/2} H^\CC_\star\\
	&= \ZZ/2[v_0,\ldots,v_n]\tensor_{\ZZ/2} H^\CC_\star.
\end{aligned}\]
(See \cite[Corollary 3.1.10]{Rav} for the computation in topology.)

This yields, for $q\equiv 1\pod{4}$, the computation
\[
	\Ext_{\mathcal{E}(n)^F}(H^F_\star,H^F_\star) = H^F_\star[v_0,\ldots,v_n]
\]
where $|v_i| = (1,(2^i-1)(1+\alpha)+1)$.

When $q\equiv 3\pod{4}$, $\mathcal{E}(n)$ does not split over $\mathcal{E}(n)^\CC$.  In order to deal with the extra complexity introduced by the relation $\tau_i^2 = \rho\tau_{i+1}$, we filter by powers of $\rho$ and consider the associated filtration spectral sequence \cite[Theorem A1.3.9]{Rav}.  (In \cite{ExtHill}, Hill refers to this spectral sequence (over $\RR$) as the ``$\rho$-Bockstein spectral sequence.")  Since $\mathcal{E}(n)^F_\star/(\rho) = \mathcal{E}(n)^\CC_\star\oplus \pi \mathcal{E}(n)^\CC_\star$, this spectral (in fact, long exact) sequence takes the form
\[
	E_1 = \left(\begin{array}{c}\Ext_{\mathcal{E}(n)^\CC}(H^\CC_\star,H^\CC_\star)\\ \oplus\\ \pi\Ext_{\mathcal{E}(n)^\CC}(H^\CC_\star,H^\CC_\star)\end{array}\right)[\rho]/(\rho^2)\implies \Ext_{\mathcal{E}(n)^F}(H^F_\star,H^F_\star).
\]
Since $\eta_R(\tau)-\eta_L(\tau) = \rho\tau_0$ in $\mathcal{E}(n)^F$, $\tau$ supports the $d_1$-differential
\[
	d_1\tau = \rho v_0.
\]
The elements $\pi$ and $\rho$ are in the Hurewicz image of the sphere, and hence are permanent cycles; the $v_i$ are represented in the cobar complex by the primitives $[\tau_i]$ and hence are permanent cycles.  Thus we have determined the $E_2$ page of the filtration spectral sequence:
\[
	E_2 = \begin{array}{c}k^M_*(F)[\tau^2,v_0,\ldots,v_n]/(\rho v_0)\\ \oplus\\ \rho\tau k^M_*(F)[\tau^2,v_0,\ldots,v_n]\end{array}.
\]
Since $\rho^2 = 0$, the spectral sequence collapses here and $E_2=E_\infty$.

In order to fully determine $\Ext_{\mathcal{E}(n)}(H_\star,H_\star)$ when $q\equiv 3\pod{4}$, we must address hidden extensions in $E_2 = E_\infty$.

\begin{prop}\label{prop:relns}
There are no hidden extensions in the $\rho$-power filtration spectral sequence for $\Ext_{\mathcal{E}(n)}$.
\end{prop}
\begin{proof}
We only need to concern ourselves with the $q\equiv 3\pod{4}$ case.  There is only one extension to consider since $\rho^2=0$.  Elements in the $\rho$-divisible summand appear in lowest possible filtration and hence have their expected multiplicative structure.  Thus it suffices to show that the $v_i$s and $\tau^2$ are free; this accomplished by considering the change-of-base map
\[
	\Ext_{\mathcal{E}(n)^F}\to \Ext_{\mathcal{E}(n)^\CC}.\qedhere
\]
\end{proof}

This, combined with the filtration spectral sequence computation, proves the following theorem.

\begin{thm}\label{thm:ExtEoo}
Over a $p$-adic field $F$,
\[
	\Ext_{\mathcal{A}_\star}(H_\star,H_\star\BPGL\langle n\rangle) =
	\begin{cases}
		k^M_*(F)[\tau,v_0,\ldots,v_n]	&\text{if }q\equiv 1\pod{4},\\ \\
		\begin{array}{c}k^M_*(F)[\tau^2,v_0,\ldots,v_n]/(\rho v_0)\\ \oplus\\ \rho\tau k^M_*(F)[\tau^2,v_0,\ldots,v_n]\end{array}	&\text{if }q\equiv 3\pod{4}.
	\end{cases}
\]\qed
\end{thm}

\begin{rmk}
Note that by Proposition \ref{prop:kMpAdic}, the algebra $k^M_*(F)$ takes the form
\[
	k^M_*(F) =
	\begin{cases}
		\ZZ/2[\pi,u]/(u^2,\pi^2)		&\text{if }q\equiv 1\pod{4},\\
		\ZZ/2[\pi,\rho]/(\rho^2,\pi(\rho-\pi))	&\text{if }q\equiv 3\pod{4}
	\end{cases}
\]
so the above computation is completely explicit.
\end{rmk}

\section{$\pi_\star\BPGL\langle n\rangle\comp$ via the motivic Adams spectral sequence}
\label{sec:motASS}
Theorem \ref{thm:HBPGL} determines the $E_2$-term of the mASS for $\BPGL\langle n\rangle$.  We now determine the mASS for $\BPGL\langle 0\rangle$ and use this and the maps $\BPGL\langle n\rangle\to\BPGL\langle 0\rangle$ to compute the mASS for $\BPGL\langle n\rangle$, $0<n\le \infty$.

Let $q$ be the residue order of our $p$-adic field $F$, let $a = \nu_2(q-1)$, and let $\lambda = \nu_2(q^2-1)$ where $\nu_2$ is the 2-adic valuation of integers.  Define the numbers $w_i$ following \cite{RW}.  Then
\[
	w_i =
	\begin{cases}
		2^{a+\nu_2(i)}	&\text{if }q\equiv 1\pod{4},\\
		2^{\lambda-1+\nu_2(i)}	&\text{if }q\equiv 3\pod{4}, i\text{ even},\\
		2			&\text{if }q\equiv 3\pod{4}, i\text{ odd}.
	\end{cases}
\]
(Note that $a=\lambda-1$ when $q\equiv 1\pod{4}$.)  By Theorem \ref{thm:BPGL0}, the following lemma is a well-known computation in \'{e}tale cohomology (see, for instance, \cite[Corollary 2.10]{RW}, apply the universal coefficient theorem, and then recall the relationship between \'{e}tale and motivic cohomology of fields).
\begin{lemma}\label{lem:HZcoeffs}
The coefficients of $\BPGL\langle 0\rangle\comp$ are
\[\begin{aligned}
	\pi_{m+n\alpha}\BPGL&\langle 0\rangle\comp\\
	&= H^{-(m+n)}_{\mathrm{mot}}(F;\ZZ_2(-n))\\
	&=
	\begin{cases}
		\ZZ_2			&\text{if }m=n=0,\\
		\ZZ_2\oplus \ZZ/w_1	&\text{if }m=0,n=-1,\\
		\ZZ/w_1			&\text{if }m=0,n=-2,\\
		\ZZ/w_i			&\text{if }m+n\alpha = (i-1)(1-\alpha)-\epsilon\alpha\\
					&\text{ for }i\ge 1, \epsilon = 1\text{ or }2,\\
		0			&\text{otherwise}.
	\end{cases}
\end{aligned}\]
\qed
\end{lemma}

\begin{thm}\label{thm:HZdiffls}
The mASS for $\BPGL\langle 0\rangle$ is determined by the following differentials:  If $q\equiv 1\pod{4}$, then
\[
	d_{a+s}\tau^{2^s} = u\tau^{2^s-1}v_0^{a+s}.
\]
If $q\equiv 3\pod{4}$, then
\[\begin{aligned}
	d_1\tau				&= \rho v_0,\\
	d_{\lambda-1+s}\tau^{2^s}	&= \rho \tau^{2^s-1} v_0^{\lambda-1+s}.
\end{aligned}\]
\end{thm}
\begin{proof}
Given Lemma \ref{lem:HZcoeffs} this is straightforward:  differentials on $\tau$-powers determine the spectral sequence since $v_0$ and elements of $k^M_*(F)$ are obviously permanent cycles.  Now $v_0$ represents $1-\epsilon$ in general, but the $0$-th coefficient group of $\BPGL\langle 0\rangle\comp$ is $\ZZ_2$ so it represents $2$ in this case.  Hence the above differentials are necessary in order to produce the appropriate $2$-torsion in $\pi_\star\BPGL\langle 0\rangle\comp$.
\end{proof}

\begin{rmk}
Upon noticing that the mASS for $\BPGL\langle 0\rangle$ is the same thing as the $2$-Bockstein spectral sequence, we can also see the differentials of Theorem \ref{thm:HZdiffls} by May's higher Leibniz rule \cite[Proposition 6.8]{LNM168}.  This states that
\[
	d_{r+1} x^2 = xd_r(x)v_0
\]
in Bockstein spectral sequences up to a correction term expressed by an (algebraic) power operation on $d_1(x)$ in the $r=1$ case.  The first nontrivial differential is determined by the fact $\BPGL\langle 0\rangle_{-\alpha} = F^\times$ \cite{MorelPi0} and the structure of $F^\times$.  When $q\equiv 1\pod{4}$, we avoid the correction term and this is enough to determine the spectral sequence.  When $q\equiv 3\pod{4}$ the correction term is quite important and implies that $d_2(\tau^2) = 0$.  One then must determine $d_?(\tau^2)$ by knowing the torsion in $\pi_{1-2\alpha}\BPGL\langle 0\rangle\comp = H^1(F;\ZZ_2(2))$, at which point one may continue with the higher Leibniz argument.
\end{rmk}

\begin{rmk}\label{rmk:diGrading}
The behavior of this spectral sequence is depicted in Figure \ref{fig:BPGL0mASS}.  Here elements in degree $(s,m+n\alpha)$ are depicted in total motivic degree $m+n\alpha-s$ with the homological degree suppressed.  (The horizontal axis measures $\ZZ$ while the vertical axis measures $\ZZ\alpha$.)  It is convenient to imagine the supressed homological degree as coming out of the page, in which case there are $v_0$-towers sitting over every mark and differentials truncate these towers.  In the dimensional range shown, $E_{\lambda+2} = E_\infty$.  Note that the $v_0$-towers in all pictures actually come out of the page, as do the $v_i$-multiplication arrows.
\end{rmk}

\begin{figure}
\centering
\includegraphics{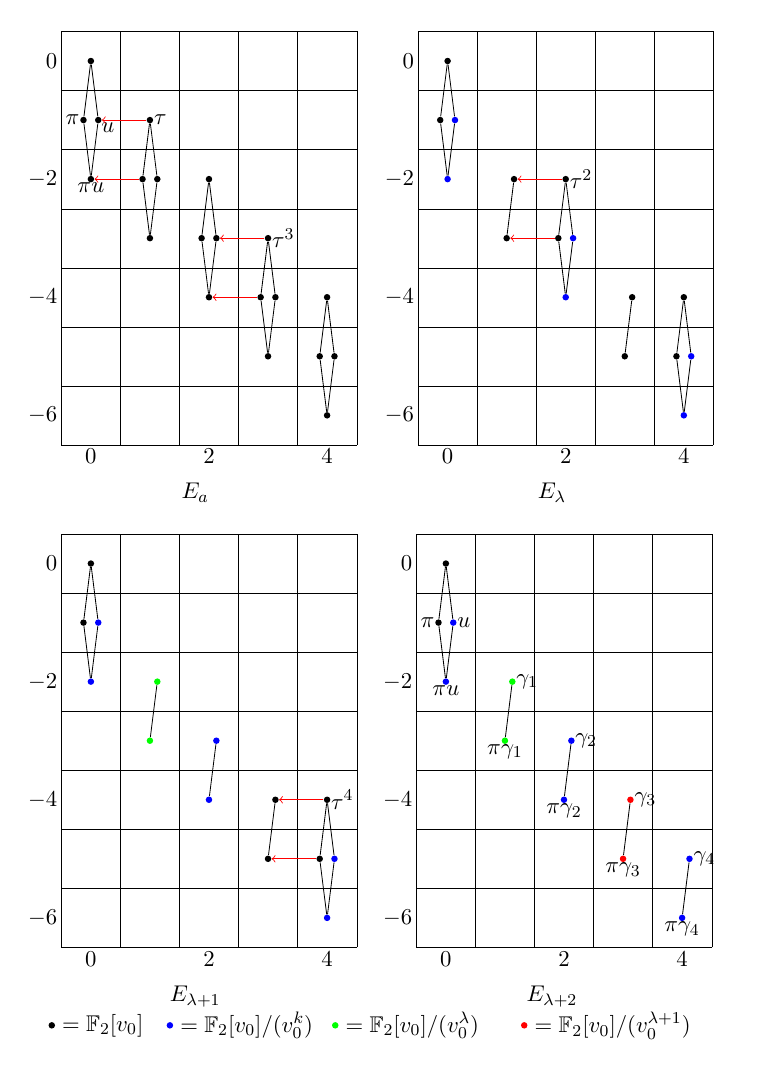}
\caption{The mASS for $\BPGL\langle 0\rangle$ over $F$}\label{fig:BPGL0mASS}
\end{figure}

We now use this information about the mASS for $\BPGL\langle 0\rangle$ to compute $\pi_\star\BPGL\langle n\rangle\comp$.  The surprising fact is that, for $i>0$, the differentials have no $v_i$-components so that $E_r(\BPGL\langle n\rangle) = E_r(\BPGL\langle 0\rangle)[v_1,\ldots,v_n]$.

\begin{thm}\label{thm:BPGLnDiffls}
The differentials in the mASS for $\BPGL\langle n\rangle$ are identical to those for $\BPGL\langle 0\rangle$ in Theorem \ref{thm:HZdiffls}.
\end{thm}
\begin{proof}
Note that all the $v_i$ and elements in $k^M_*(F)$ are permanent cycles.  (This is obvious since their target ranges are trivial; see Figure \ref{fig:BPGLmASS}.)  Inductively, then, we can show that
\[
	E_r(\BPGL\langle n\rangle) = E_r(\BPGL\langle 0\rangle)[v_1,\ldots,v_n].
\]
The result is true on $E_2$ pages; assume it is true on $E_r$ and consider the map $f$ of spectral sequences induced by $\BPGL\langle n\rangle\to \BPGL\langle 0\rangle$.  By dimensional accounting, it is clear that the $E_{r+1}$-page is determined by the differential on the smallest surviving $\tau$-power.  The map $f$ is an isomorphism on the target of this $\tau$-power by dimensional accounting.  Indeed, since $cd_2(F) = 2$ (i.e. $k^M_*(F)$ is only nonzero for $0\le *\le 2$) no $v_{>0}$-terms can appear in the target.  This proves that the $\tau$-power supports the same differential as in the mASS for $\BPGL\langle 0\rangle$, from which we get the result on $E_{r+1}$.
\end{proof}

\begin{rmk}
The mASS for $\BPGL\langle n\rangle$ is depicted in Figure \ref{fig:BPGLmASS}.  The grading convention described in Remark \ref{rmk:diGrading} is followed again.  Note that the arrows with slope 1 represent multiplication by $v_i$'s, $i>0$, and they also come ``out of the page" since $v_i$ has homological degree 1.  Note the obvious vanishing region with $k^M_*(F)$ and the $v_i$ monomials on the boundary.
\end{rmk}

\begin{figure}
\centering
\includegraphics{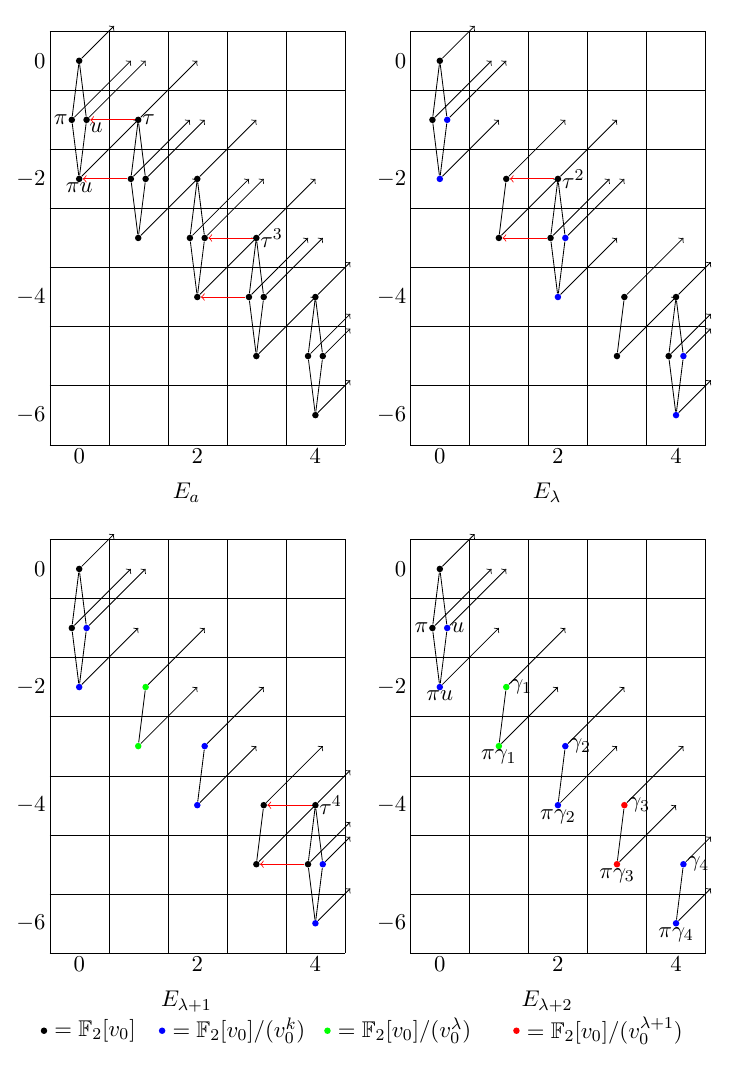}
\caption{The mASS for $\BPGL\langle n\rangle$ over $F$}\label{fig:BPGLmASS}
\end{figure}

We now have the following unified description of the $E_\infty$ page of the mASS for $\BPGL\langle n\rangle$, $0\le n\le \infty$.

\begin{thm}\label{thm:EooBPGL}
Let $a = \nu_2(q-1)$ and let $\lambda = \nu_2(q^2-1)$.  The $E_\infty$-term of the mASS for $\BPGL\langle n\rangle$ over a $p$-adic field $F$ is
\[
	\Gamma'[v_1,\ldots,v_n]
\]
where $\Gamma'$ has additive structure
\[
	\Gamma' =
	\begin{cases}
		\ZZ/2[v_0]				&\text{in dimension }0,\\
		\ZZ/2[v_0]\braces{\pi}\oplus \ZZ/2[v_0]/v_0^{a}\braces{u}	&\text{in dimension }-\alpha,\\
		\ZZ/2[v_0]/v_0^a\braces{\pi u}			&\text{in dimension }-2\alpha,\\
		\ZZ/2[v_0]/v_0^{\lambda-1+\nu_2(i)}\braces{\gamma_i}		&\text{in dimension }i(1-\alpha)-\alpha\\ &\text{for }i\text{ odd},\\
		\ZZ/2[v_0]/v_0^{\lambda-1+\nu_2(i)}\braces{\pi\gamma_i}	&\text{in dimension }i(1-\alpha)-2\alpha\\	&\text{for }i\text{ odd},\\
		\ZZ/2[v_0]/v_0^a\braces{\gamma_i}		&\text{in dimension }i(1-\alpha)-\alpha\\ &\text{for }i\text{ even},\\
		\ZZ/2[v_0]/v_0^a\braces{\pi\gamma_i}	&\text{in dimension }i(1-\alpha)-2\alpha\\ &\text{for }i\text{ odd},\\
		0			&\text{otherwise}.
	\end{cases}
\]
The generators in each degree are indicated above.  They satisfy the obvious multiplicative relations indicated by their notation while $u\gamma_i$ and $\gamma_i\gamma_j$ are 0.
\qed
\end{thm}

A quick inspection of tri-degrees reveals that there are no extensions except those created by $v_0$-multiplication.  Indeed, the lines of slope 1 originating in the nontrivial dimensions of $\Gamma'$ do not overlap, so we only need to worry about $v_0$-towers.  Since $v_0$ represents 2 in $\pi_0 \BPGL\langle n\rangle\comp$, any copies of $\ZZ/2[v_0]$ produce copies of the 2-adic integers $\ZZ_2$, and any copies of $\ZZ/2[v_0]/v_0^{b}$ produce copies of $\ZZ/2^{b}$.  This proves the following theorem.

\begin{thm}\label{thm:BPGL}
Let $a=\nu_2(q-1)$, $\lambda=\nu_2(q^2-1)$ and set $w_i = \ZZ/2^a$ for $i$ odd, $w_i = 2^{\lambda-1+\nu(i)}$ for $i$ even.  The coefficients of the 2-complete algebraic Johnson-Wilson spectra $\BPGL\langle n\rangle\,\widehat{_2}$ over a $p$-adic field $F$ are
\[
	\pi_\star \BPGL\langle n\rangle\,\widehat{_2} = (\pi_\star\BPGL\langle 0\rangle\comp)[v_1,v_2,\ldots]
\]
where $|v_i| = (2^i-1)(1+\alpha)$ and, additively,
\[
	\pi_\star\BPGL\langle 0\rangle\comp =
	\begin{cases}
		\ZZ_2			&\text{in dimension }0,\\
		\ZZ_2\oplus \ZZ/w_1	&\text{in dimension }-\alpha,\\
		\ZZ/w_1			&\text{in dimension }-2\alpha,\\
		\ZZ/w_i			&\text{in dimension }(i-1)(1-\alpha)-\epsilon\alpha\\
					&\text{ for }i\ge 1\text{ and }\epsilon = 1\text{ or }2,\\
		0			&\text{otherwise}.
	\end{cases}
\]
Multiplicative structure and generator names for $\pi_\star\BPGL\langle 0\rangle\comp$ are the same as in Theorem \ref{thm:EooBPGL}.
\qed
\end{thm}

\begin{cor}\label{cor:MGL}
The coefficients of the 2-complete algebraic cobordism spectrum over a $p$-adic field $F$ are
\[
	\pi_\star\MGL\,\widehat{_2} = \Gamma[x_1,x_2,\ldots]
\]
where $|x_i| = i(1+\alpha)$.
\qed
\end{cor}

\begin{cor}
Over a $p$-adic field, the slice spectral sequences for $\MGL\comp$ and $\BPGL\langle n\rangle\comp$ ($0\le n\le \infty$) collapse.
\end{cor}
\begin{proof}
Hopkins-Morel show that the slice associated gradeds for $\MGL$ and $\BPGL\langle n\rangle$ are $H\ZZ_\star[x_1,x_2,\ldots]$, $H\ZZ_\star[v_1,\ldots,v_n]$, respectively, and Theorem \ref{thm:BPGL} implies that there are no differentials in the slice spectral sequences for their 2-completions.
\end{proof}

\begin{cor}\label{cor:kgl}
The coefficients of 2-complete $\kgl$ are $(\pi_\star\BPGL\langle 0\rangle)[v_1]$.  In particular, there is no $v_1$-torsion and we recover the 2-complete algebraic $K$-theory of $F$ in degrees $m+0\alpha$.
\end{cor}

We conclude by discussing some easy corollaries of this work that highlight the importance of the algebraic Brown-Peterson spectra and have important applications to the motivic ANSS \cite{thesis, motAlpha}.  See \cite{HKO,KOAdams,DI} for discussions of the motivic Adams-Novikov spectral sequence.

For convenience, let $\Gamma \coleq \pi_\star\BPGL\langle 0\rangle\comp = \pi_\star H\ZZ_2$.  For typographical simplicity, we drop the 2-completion $(~)\,\widehat{_2}$ from our notation in the rest of this section.

\begin{thm}\label{thm:ANSSE2}
Fix a $p$-adic field $F$ and work in the 2-complete stable motivic homotopy category over $F$.  Then the Hopf algebroid for $\BPGL$ splits as
\[
	(\BPGL_\star, \BPGL_\star\BPGL) = (\BP_*,\BP_*\BP)\tensor_{\ZZ_2} \Gamma.
\]
Moreover, the $E_2$-term of the motivic ANSS in homological degree $s$ is
\[
	\Ext^s_{\BPGL_\star\BPGL}(\BPGL_\star,\BPGL_\star) = \begin{array}{c}{}^\top E_2^s\tensor_{\ZZ_2} \Gamma\\ \oplus\\ \Tor^{\ZZ_2}_1({}^\top E_2^{s+1}, \Gamma)\end{array}.
\]
Here ${}^{\top}E_2 = \Ext_{\BP_*\BP}(\BP_*,\BP_*)$ with degrees shifted so that elements appearing in degree $(s,2m)$ in topology appear in degree $(s,m(1+\alpha))$ motivically.
\end{thm}
\begin{proof}
The second statement is an easy consequence of the first via the cobar resolution computing $\Ext_{\BPGL_\star\BPGL}(\BPGL_\star,\BPGL_\star)$ and the universal coefficient theorem.

As a consequence of motivic Landweber exactness, Naumann-{\O}stv{\ae}r-Spitzweck \cite{MLE} deduce a splitting of the $\MGL$ Hopf algebroid as
\[
	(\MGL_\star,\MGL_\star\MGL) = (\MU_*,\MU_*\MU)\tensor_{\MU_*} \MGL_\star
\]
where $\MU_*$ is the coefficients of (topological) complex cobordism, the Lazard ring.  This splitting passes to $\BPGL$, so
\[
\begin{aligned}
	(\BPGL_*,\BPGL_*\BPGL)	&= (\BP_*,\BP_*\BP)\tensor_{\BP_*} \BPGL_\star\\
				&= (\BP_*,\BP_*\BP)\tensor_{\BP_*} (\BP_*\tensor_{\ZZ_2} \Gamma)\\
				&= (\BP_*,\BP_*\BP)\tensor_{\ZZ_2} \Gamma.
\end{aligned}
\]
\end{proof}

This description of the $E_2$-term of the motivic ANSS over $F$ already pays dividends in the form of a graded algebra with infinitely many nonzero components previously undiscovered in the stable stems $\pi_\star\mathbbm{1}\comp$ of the 2-complete sphere spectrum.

\begin{thm}\label{thm:GammaSurvives}
The algebra $\Gamma$ survives to $E_\infty$ of the motivic Adams-Novikov spectral sequence and represents a copy of $(H\ZZ_2)_\star$ in $\pi_\star\mathbbm{1}\comp$.
\end{thm}
\begin{proof}
The elements of $\Gamma = E_2^{0,0}\tensor \Gamma$ are in filtration 0 and hence are not the targets of differentials.  We must show that elements of $\Gamma$ do not support differentials.  For an element of degree $(s,m+n\alpha)$ in $E_2$, call $m+n-s$ the classical Adams degree.  Since $\Gamma$ is concentrated in classical Adams degrees $0, -1$, and $-2$, Theorem \ref{thm:ANSSE2} gives a vanishing line $E_2^{s,m+n\alpha} = 0$ for $s>m+n-s+2$.  Differentials in the motivic ANSS decrease classical Adams degree by 1 and increase homological degree by at least 2.  Since $\Gamma$ has $s=0$ and classical Adams degree $0,-1$, and $-2$, we see that it does not support differentials.
\end{proof}

\bibliography{biblio}

\end{document}